\documentclass[12pt]{article}
\usepackage{ayv_paperstyle}

\usepackage[all,cmtip]{xy}

\title{Control theory approach to continuous-time finite state mean field games}
\author{Yurii Averboukh\footnote{Higher School of Economics, Moscow, Russia, \email{averboukh@gmail.com}}}
\date{}

\begin{document}
\maketitle

\begin{abstract}
In the paper, we use the equivalent formulation of a finite state mean field game as a control problem with mixed constraints to study the dependence of  solutions to  finite state mean field game on  an initial distribution of players. We introduce the concept of  value multifunction of the mean field game that is a mapping assigning to an initial time and an initial distribution a set of expected outcomes of the representative player corresponding to solutions of the mean field game. Using the control reformulation of the finite state mean field game, we give the sufficient condition on a given multifunction to be a value multifunction in the terms of the viability theory. The maximal multifunction (i.e. the mapping assigning to an initial time and distribution the whole set of values corresponding to solutions of the mean field game) is characterized via the backward attainability set for the certain control systems.
\msccode{49N80, 91A16, 
	49J45, 60J27, 34H05	
}
\end{abstract}

\section{Introduction}
The mean field game theory examines  systems of identical players those interacts via some external media. It was proposed independently by Lasry, Lions~\cite{Lasry_Lions_2006_I},~\cite{Lasry_Lions_2006_II} and Huang, Caines, Malham\'{e} ~\cite{Huang_Caines_Malhame_2007},~\cite{Huang_Malhame_Caines_2006}. In the paper, we study the continuous-time finite state mean field games that is the infinite  player dynamic game under assumptions that the players are similar, when the dynamics of each players is given by a continuous-time finite state Markov chain with transition probabilities depending on distribution of all players and player's actions.   The finite state mean field games find various application in the analysis of socio-econimic systems and modeling of cybersecurity~\cite{Kolokoltsov_security},~\cite{Kolokoltsov_Malofeev_2},~\cite{Kolokoltsov_Yang},~\cite{Kolokoltsov_Bensoussan},~\cite{Kolokoltsov_Malofeev_book}. 

The study of the continuous-time finite state mean field games started with papers~\cite{Basna_Hilbert_Kolokoltsov},~\cite{Gomes_Mohr_Souza_finite_state}. In those papers the approach based on the mean field game system consisting of Bellman  and kynetic equations was developed. Notice that, for the   finite state mean field game, the mean field game system is a system of ODEs with mixed (initial and boundary) conditions. The results of the aforementioned papers were extended to the case of finite state mean field games with common noise acting in a finite sequence of times in~\cite{finite_state_common_noise}.  

The further progress of the theory of finite state mean field games is due to the probabilistic approach~\cite{Cecchin_Fischer_probabilistic_finite_state} and master equation ~\cite{Bayraktar_Ceccin_et_al_common_noise},~\cite{Bayraktar_et_al_Wright_N_player},~\cite{Bayraktar_Cohen_master},~\cite{master_Cechin_Pelino}. The master equation first proposed by Lions~\cite{lions_lecture} for the second-order mean field games is used to justify the convergence of feedback equilibria in the finite player games to the solution of the mean field game and to study the dependence of the solution of the mean field game on the initial distribution of players. However, the theory of  master equation for the mean filed games is developed in the case of unique solution to the mean field game~\cite{Bayraktar_et_al_Wright_N_player},  ~\cite{Bayraktar_Cohen_master},~\cite{master_Cechin_Pelino} or in the presence of the common noise~\cite{Bayraktar_Ceccin_et_al_common_noise}. 

In the paper, we study the dependence of  solutions of the continuous-time finite state mean field game in the general case without any uniqueness assumptions. To this end, we express the  mean field game as a finite-dimensional control problem with mixed constraints and  use the concept  of  value multifunction that is a set-valued mapping assigning to an initial time and an initial distribution of players a set of expected outcomes of the representative player. The maximal value multifunction comprises all solutions of the mean field game. We find the sufficient condition on a given multifunction to be a value multifunction in the terms of the viability theory. Furthermore, we characterize the maximal value function using the attainability domain for the certain dynamical system in the backward time constructed by the original mean filed game. Finally, we show that a graph of smooth solution of the master equation is a value multifunction.

The paper is organized as follows. In Section~\ref{sect:finite_state_mfg}, we introduce the main concepts of the continuous-time finite state mean field games. Section~\ref{sect:relaxation} is concerned with the relaxation of the control space. The control problem equivalent to the finite state mean field game is introduced in Section~\ref{sect:optimal_control}. The study of dependence of the solution of the finite state mean field game on the initial distribution is studied in Section~\ref{sect:dependence}. Here we introduce the mean field game dynamical system and show that a multifunction satisfying the viability condition for this dynamical system is a value mutifunction. Furthermore, we prove that the maximal value multifunction can be found using the attainability domain for the certain dynamical system. The last result of Section \ref{sect:dependence} gives the link between the  solution of master equation and the value multifunction.  

\section{Problem setting}\label{sect:finite_state_mfg} In this section, we discuss the mean field game system for the finite state case and deduce the master equation in the multivalued form.  Since we do not impose condition ensuring the uniqueness of the optimal control of the representative player,  the strong analysis of the finite state  mean filed  involves the relaxation of the control space which is introduced in Section \ref{sect:relaxation}.
\subsection{Mean field game system for finite state space}\label{subsect:mfg_system}
Without loss of generality, we assume the state space is $\{1,\ldots,d\}$, where $d\in \mathbb{N}$. Distributions on the space $\{1,\ldots,d\}$ lie in the $d$-dimensional simplex $$\Sigma^d\triangleq\{(m_1,\ldots,m_d):\,m_i\geq 0,\ \ m_1+\ldots+m_d=1\}.$$  In the following, we regard $m\in\Sigma^d$ as a row-vector and endow $\Sigma^d$ with the Euclidean metric. Additionally, we embed $\Sigma^d$ into the Euclidian space of row-vectors denoted below by $\mathbb{R}^{d*}$. 

Recall that  mean field games examine  systems of infinitely many similar players. In the finite state case (see \cite{Basna_Hilbert_Kolokoltsov}, \cite{Gomes_Mohr_Souza_finite_state}), the dynamics of each player is given by the Markov chain with the Kolmogorov matrix depending on the current distribution of agents 
$$Q(t,m,u)=(Q_{i,j}(t,m,u))_{i,j=1}^d. $$
Here, $t$ stands for a time, $m$ is a distribution of all player, $u$ is a control of the player chosen from some set $U$. We consider the mean field game on the finite time horizon $[t_0,T]$ and assume that each player tries to maximize 
\begin{equation}\label{intro:object_f}
\mathbb{E}\left[\sigma(X(T),m(T))+\int_{t_0}^T g(t,X(t),m(t),u(t))dt\right],
\end{equation}
where $X$ is the stochastic process describing the motion of the player, $m(t)$ is the distribution of players at time $t$, whereas $u(t)$ denotes the instantaneous control of the player.

We impose the following conditions on the control space, the Kolmogorov matrix and the payoff functions.
\begin{enumerate}[label=(C\arabic*)]
	\item The set $U$ is a metric compact.
	\item The functions $Q_{i,j}$, $\sigma$ and $g$ are continuous.
	\item The dependence of $Q_{i,j}$ and $g$ on $m$ is Lipschitz continuous uniformly w.r.t. $t$ and $u$.
\end{enumerate}

The solution of the mean field game is given by the common solution of two problems. The first one describes the dynamics of the distribution of players, when the second problem is the optimization problem for the representative player. To introduce them, let us describe the strategies used by the players. 

We assume that the players are informed about current time and state. Since the state space is finite, one can assume that, a policy is now a $d$-tuple $u(\cdot)=(u_i(\cdot))_{i=1}^d$, where $u_i(\cdot)$, $i=1,\ldots,d$, are measurable functions from $[t_0,T]$ to $U$. Below, we assume that the player who occupies the state $i$ at time $t$ uses the control $u_i(t)$. Therefore, the dynamics of distribution of players obeys the following Kolmogorov equation:
\begin{equation}\label{eq:system_KOlmogorov}
\frac{d}{dt}m_j(t)=\sum_{i=1}^{d}m_i(t)Q_{i,j}(t,m(t),u_i(t)),\ \ j=1,\ldots,d,\ \ m(t_0)=m_0.
\end{equation} It is convenient to rewrite this system in the vector form. To this end, given a vector of controls $u=(u_i)_{i=1}^d\in U^d$ and $m\in \Sigma^d$,  denote the matrix with the elements $Q_{i,j}(t,m,u_i)$ by $\mathcal{Q}(t,m,u)$. Notice that now $U^d$ serves as a new control space. Therefore, system~(\ref{eq:system_KOlmogorov}) can be rewritten as follows:
\begin{equation}\label{eq:motion_usual_strategies}
\frac{d}{dt}m(t)=m(t)\mathcal{Q}(t,m(t),u(t)),\ \ m(t_0)=m_0.
\end{equation}

The second problem is the optimization problem for each player. Given a flow of probabilities $m(\cdot)$, we obtain the Markov decision problem with the dynamics determined by the Kolmogorov matrix
$$Q(t,m(t),u) $$ and objective function~(\ref{intro:object_f}). Recall that we use the feedback strategies. Within this framework, the solution of the optimal control problem can be obtained  from the  dynamic programming. To introduce it for the Markov decision problem, we define the Hamiltonian by the following rule: for $t\in [0,T]$, $m\in\Sigma^d$, $\phi\in \rd$, set
\begin{equation}\label{intro:H_i}
H_i(t,m,\phi)\triangleq \max_{u\in U}\left[\sum_{j=1}^dQ_{i,j}(t,m,u)\phi_j+g(t,i,m,u)\right], 
\end{equation}
$$H(t,m,u)=(H_i(t,m,u))_{i=1}^d. $$ 
For any flow of probabilities is $m(\cdot)$, the value function $\phi:[0,T]\rightarrow\rd$ satisfies the Bellman equation
\begin{equation}\label{eq:Zachrison_usual}
\frac{d}{dt}\phi(t)=-H(t,m(t),\phi(t)),\ \ \phi_i(T)=\sigma(i,T), 
\end{equation}
whereas the feedback optimal strategy is computed by the rule
\begin{equation}\label{incl:u_optimal}
\hat{u}_i(t)\in \underset{u\in U}{\operatorname{Argmax}}\left[\sum_{j=1}^dQ_{i,j}(t,m,u)\phi_j(t)+g(t,i,m(t),u)\right]. 
\end{equation} A solution of the system~(\ref{eq:motion_usual_strategies}), ~(\ref{eq:Zachrison_usual}),~(\ref{incl:u_optimal}) is  a solution of the finite state mean field game. 

\subsection{Master equation in the multivalued form}\label{subsect:master_eq}
	To study the dependence of the solution of the mean field game on the initial distribution of players, Lions proposed the concept of the master equation~\cite{lions_lecture}. The equivalence between the classical formulation and master equation for the second order case is shown in \cite{Master_cdll}. Moreover, there the master equation technique was used to deduce the convergence of solutions of $N$-player stochastic differential games to the solution of the mean field game. These results were extended to the finite state mean field games in~\cite{Bayraktar_Cohen_master},~\cite{master_Cechin_Pelino}. The assumptions of the aforementioned works include the uniqueness of the control maximizing the expression of the Hamiltonian (i.e. the Hamiltonian is assumed to be strictly convex). In the paper, we do not impose this condition. Thus, it is reasonable to introduce the master equation  in the multivalued setting. In Proposition \ref{prop:value_to_master} below, we derive this equation under assumption that there exists a smooth value function of the mean field game. Proposition \ref{prop:master} in Section \ref{sect:dependence} states that each smooth solution of the master equation in the multivalued form provides the solution of the mean field game system.
	
	The main object of the master equation is the  function $\Phi:[0,T]\times\Sigma^d\rightarrow\rd$ such that the quantity  $\Phi_i(s,\mu)$ can be interpreted as the optimal outcome of the representative player who starts at the time $s$ from the state~$i$ under assumptions that the initial distribution of agents is $\mu$ and all players use an optimal strategy realizing a solution of the MFG.  Below, $\partial\Phi/\partial s$ denotes the derivative of $\Phi$ w.r.t. time variable, while ${\partial \Phi}/{\partial \mu}$ stands for the derivative of $\Phi$ w.r.t. measure, i.e., given $s\in [0,T]$ and $\mu\in\Sigma^d$, ${\partial \Phi}/{\partial \mu}(s,\mu)$ is a $d\times d$ matrix such that, for every $\mu'\in\Sigma^d$,
	\begin{equation}\label{intro:derivative_Phi_measure}\Phi(s,\mu')-\Phi(s,\mu)=(\mu'-\mu)\frac{\partial\Phi}{\partial\mu}(s,\mu)+o(\|\mu'-\mu\|).\end{equation}
	
	To introduce the master equation, for $\phi\in \rd$,  put
	\[\begin{split}
	\mathcal{O}(t,m,\phi)\triangleq \Bigl\{m\mathcal{Q}&(t,m,u):u=(u_i)_i^d,\\ &\sum_{j=1}^dQ_{i,j}(t,m,u_i)\phi_j(t)+g(t,i,m(t),u_i)=H_i(t,m,\phi)\Bigr\}.\end{split}\]

	\begin{proposition}\label{prop:value_to_master}
		Let $\Phi:[0,T]\times\Sigma^d\rightarrow\rd$ be a continuously differentiable function such that, for every $t_0\in [0,T]$, $m_0\in\Sigma^d$, there exists a solution to mean field game system~(\ref{eq:motion_usual_strategies}),~(\ref{eq:Zachrison_usual}),~(\ref{incl:u_optimal}) $(\phi(\cdot),m(\cdot))$ such that $\Phi(t,m(t))=\phi(t)$ for $t\in [t_0,T]$ and $m(t_0)=m_0$. Then,
		$\Phi$ satisfies the master equation in the multivalued form
		\begin{equation}\label{eq:master_usual}
		\frac{\partial}{\partial s}\Phi(s,\mu)+H(t,\mu,\Phi(s,\mu))\in - \mathcal{O}(s,\mu,\Phi(s,\mu))\frac{\partial \Phi}{\partial \mu}(s,\mu).
		\end{equation}     
	\end{proposition}
	\begin{proof}
	By assumption the function $t\mapsto\Phi(t,m(t))$  solves Bellman equation~(\ref{eq:Zachrison_usual}), when $m(\cdot)$ satisfies~(\ref{eq:system_KOlmogorov}) for the strategy $\hat{u}$ obeying~(\ref{incl:u_optimal}). The later assumption implies the inclusion
	\begin{equation}\label{incl:m_O}
	\frac{d}{dt}m(t)\in\mathcal{O}(t,m(t),\Phi(t,m(t))).
	\end{equation} Differentiating  $\Phi(t,m(t))$, from the Bellman equation and the definition of the derivative of the function $\Phi$ w.r.t. measure (see (\ref{intro:derivative_Phi_measure})), we obtain
	\[\frac{\partial \Phi(t,m(t))}{\partial t}+H(t,m(t),\Phi(t,m(t)))+\frac{dm(t)}{dt}\cdot \frac{\partial\Phi(t,m(t))}{\partial \mu}=0.\] Using this,~(\ref{incl:m_O}) and plugging in the  resulting inclusion only for the initial position $(s,\mu)$, we arrive at  master equation in the multivalued form~(\ref{eq:master_usual}).
\end{proof}
	
	\begin{remark}If, for each $t\in [0,T]$, $m\in\Sigma^d$ and $\phi\in\rd$, there exists a unique $\bar{u}=(\bar{u}_i)_{i=1}^d$ such that
	\[\begin{split}
	\sum_{j=1}^dQ_{i,j}&(t,m,u_i)\phi_j+g(t,i,m(t),u_i)\\=&\sum_{j=1}^dQ_{i,j}(t,m,\bar{u}_i)\phi_j+g(t,i,m(t),\bar{u}_i),
	\end{split}\] the Hamiltonian is differentiable w.r.t. $\phi$, and master equation (\ref{eq:master_usual}) takes the more familiar single-valued form 
	\[\frac{\partial}{\partial s}\Phi(s,\mu)+H(t,\mu,\Phi(s,\mu)) +\frac{\partial H(t,\mu,\phi)}{\partial \phi}\frac{\partial \Phi}{\partial \mu}(s,\mu)=0.\]\end{remark}

\section{Relaxation of the finite state mean field game}\label{sect:relaxation}

In the previous section, we assumed that  players  occupying the same state use the same control. However, in the case when the Hamiltonian is not strictly convex, the optimal control is nonunique. Furthermore, the set of feedback strategies is neither convex nor closed. To overcome this difficulties, we use the relaxed (randomized) feedback strategies~\cite{Hermandez_Lerma_Guo}.

A  feedback strategy is  a mapping $$t\mapsto \nu(t)\triangleq (\nu_1(t,du),\ldots,\nu_d(t,du)) $$ such that
\begin{itemize}
	\item $\nu_i(t,\cdot)$ is a probability on $U$;
	\item the dependence $t\mapsto \int_U \zeta(u)\nu_i(t,du)$ is measurable for any continuous function $\zeta:U\rightarrow \mathbb{R}$.  
\end{itemize}

The relaxed strategies imply that the players occupying the state $i$ at the time $t$ distribute their controls according to the probability $\nu_i(t,du)$. 

Thus, for a fixed flow of probabilities $m(\cdot)$ and a  relaxed strategy $\nu(\cdot)$, the motion of a representative player is the Markov chain with the Kolmogorov matrix
$$\left(\int_U Q(t,m(t),u)\nu_i(t,du)\right)_{i,j=1}^d. $$ With some abuse of notation, for $t\in [0,T]$, $m\in\Sigma^d$, $\nu=(\nu_1,\ldots,\nu_d)\in (\mathcal{P}(U))^d$,  we set
$$\mathcal{Q}_{i,j}(t,m,\nu)\triangleq \left(\int_U Q_{i,j}(t,m,u)\nu_i(du)\right), $$
$$g_i(t,m,\nu)=g(t,i,m,\nu)\triangleq \int_{U}g(t,i,m,u)\nu_i(du). $$ Here and below $\mathcal{P}(U)$ stands for the set of all Borel probabilities on $U$.

Furthermore, notice that
$$H_i(t,m,\phi)=\max_{\nu_i\in\mathcal{P}(U)}\left[\sum_{j=1}^q \int_U Q(t,m,u)\nu_i(du)\cdot \phi_j+g_i(t,m,\nu_i)\right]. $$ Finally, we introduce the matrix and vector notation:
\[\mathcal{Q}(t,m,\nu)=(\mathcal{Q}_{i,j}(t,m,\nu))_{i,j=1}^d.\]
$$H(t,m,\phi)\triangleq (H_1(t,m,\phi),\ldots,H_d(t,m,\phi))^T,$$ $$g(t,m,\nu)=(g_1(t,m,\nu),\ldots,g_d(t,m,\nu))^T, $$ $$\sigma(m)\triangleq (\sigma(1,m),\ldots,\sigma(d,m))^T.$$

Using this notation, we give the following definition.

\begin{definition}\label{def:control_process} Let $t_0\in [0,T]$ be an initial time, $m_0\in\Sigma^d$ be an initial distribution of players. We say that a pair $(\phi(\cdot),m(\cdot))$ is a solution of the mean field game if there exists a  relaxed feedback strategy $\hat{\nu}(\cdot)=(\hat{\nu}_1(\cdot),\ldots,\hat{\nu}_d(\cdot))$ such that
	\begin{enumerate}
		\item $m(\cdot)$ satisfies the Kolmogorov equation
		\begin{equation}\label{eq:motion_relaxed_star}
		\frac{d}{dt}m(t)=m(t)\mathcal{Q}(t,m(t),\hat{\nu}(t)),\ \ m(t_0)=m_0.
		\end{equation}
		\item $\phi(\cdot)$ satisfies  Bellman equation~(\ref{eq:Zachrison_usual}), i.e.,
		\begin{equation}\label{eq:HJ_def}
		\frac{d}{dt}\phi(t)=-H(t,m(t),\phi(t)),\ \ \phi(T)=\sigma(m(T)) 
		\end{equation}
		\item $\hat{\nu}$ is an optimal control, i.e., for a.e. $t\in [t_0,T]$,
		\begin{equation}\label{incl:hat_nu}
		\hat{\nu}_i(t) \in \underset{\nu_i\in \mathcal{P}(U)}{\operatorname{Argmax}}\left[\sum_{j=1}^q \int_U Q(t,m,u)\nu_i(du)\cdot \phi_j+\int_Ug(t,i,m,u)\nu_i(du)\right] 
		\end{equation} or, equivalently,
		$$\frac{d}{dt}\phi(t)+\mathcal{Q}(t,m(t),\hat{\nu}(t))\phi(t)+g(t,m(t),\hat{\nu}(t))=0. $$
	\end{enumerate}
	
\end{definition}

\begin{theorem}\label{th:existence}
	Under assumptions (C1)--(C3), there exists at least one solution of the mean field game.
\end{theorem}
The proof is given in the Appendix. It relies on the fixed point arguments and representation of the feedback strategies by control measures.

Below we also will use the probabilistic representation of the finite state mean field game. It is convenient to introduce the dynamics of the representative player using generator. Let $t\in [0,T]$, $i\in \{1,\ldots,d\}$, $m\in\Sigma^d$, $\nu\in(\mathcal{P}(U))^d$. Define the generator  $\Lambda_t[m,\nu]$ that is a linear operator  on $\rd$ by the following rule: for $\psi\in \mathbb{R}^d$,  $\Lambda_t[m,\nu]\psi$ is the vector with coordinates \[(\Lambda_t[m,\nu]\psi)_i \triangleq \sum_{j=1}^d \mathcal{Q}_{i,j}(t,m,\nu)\psi(j),\ \ i=1,\ldots,d. \]  Given $s,r\in [0,T]$, $s<r$, a flow of probabilities $m(\cdot)$ and a relaxed feedback strategy $\eta(\cdot)$, we say that the 5-tuple $(\Omega,\mathcal{F},\{\mathcal{F}_t\}_{t\in [s,r]},P,X)$, where $(\Omega,\mathcal{F},\{\mathcal{F}_t\}_{t\in [s,r]},P)$ is a filtered probability space, and $X$ is a $\{\mathcal{F}_t\}_{t\in [s,r]}$-adapted stochastic process taking values in $\{1,\ldots,d\}$, provides a motion of the representative player on $[r,s]$ if, for any $\psi\in\mathbb{R}^d$, the process
\[\psi_{X(t)}-\int_{s}^{t}(\Lambda_\tau[m(\tau),\eta(\tau)]\psi)_{X(\tau)}d\tau\] is a  $\{\mathcal{F}_t\}_{t\in [s,r]}$-martingale.

Notice that, for every initial probability on $\{1,\ldots,d\}$, at least one  motion of the representative player always exists~\cite[Theorem 3.11]{Kol_book} (see also~\cite[Propositions 4.1 and Example 4.5]{Kolokoltsov_Li_Yang_2011}). 
One can introduce the probabilities  $\mu_i(t)\triangleq P(X(t)=i)$. If $\mu_*=(\mu_{*,i})_{i=1}^d$ is the initial distribution at the time $s$, then the vector $\mu(t)=(\mu_i(t))_{i=1}^d$ obeys the Kolmogorov equation
\begin{equation}\label{eq:mu}\frac{d}{dt}\mu(t)=\mu(t)\mathcal{Q}(t,m(t),\eta(t)),\ \ \mu(s)=\mu_*.\end{equation}

Now let us reformulate Definition \ref{def:control_process} within the framework of the probabilistic approach.

\begin{proposition}\label{prop:equiv:probability}The pair $(\phi(\cdot),m(\cdot))$ solves the mean field game with the equilibrium feedback strategy $\hat{\nu}(\cdot)$, if and only if
	(\ref{eq:motion_relaxed_star}),~(\ref{eq:HJ_def}) holds and $\hat{\nu}(\cdot)$ is the optimal feedback relaxed strategy at any initial distribution for the Markov decision problem  on $[t_0,T]$
	\begin{equation}\label{criterion:probability}
	\text{maximize }\mathbb{E}\left[\sigma(X(T),m(T))+\int_{t_0}^Tg(t,X(t),m(t),\eta(t))dt\right]
	\end{equation}
	subject to a feedback relaxed strategy $\eta(\cdot):[t_0,T]\rightarrow (\mathcal{P}(U))^d$ and a corresponding motion of the representative player $(\Omega,\mathcal{F},\{\mathcal{F}_t\}_{t\in [t_0,T]},P,X)$ at any initial position, where $\mathbb{E}$ stands for the expectation according to the probability $P$.
\end{proposition}
\begin{proof}
	It suffices to prove that the fact that $\hat{\nu}$ is optimal control for the Markov decision problem on $[t_0,T]$ with payoff given by~(\ref{criterion:probability}) subject to a strategy $\eta$ and a corresponding motion of the representative player $(\Omega,\mathcal{F},\{\mathcal{F}_t\}_{t\in [t_0,T]},P,X)$ is equivalent to~(\ref{incl:hat_nu}). This directly follows from the dynamic programming principle.
\end{proof}
Notice that the expectation of payoff~(\ref{criterion:probability}) is determined only by $\eta(\cdot)$, $m(\cdot)$ and the initial distribution $\mu_0$ equation (\ref{eq:mu}) with $s=t_0$ and $\mu_*=\mu_0$  by the formula:
\begin{equation}\label{equality:exp_payoff_mu}
\begin{split}
\mathbb{E}\Bigl[\sigma(X(T),&m(T))+\int_{t_0}^Tg(t,X(t),m(t),\eta(t))dt\Bigr]\\&=\mu(T)\sigma(m(T))+\int_{t_0}^T\mu(t)g(t,m(t),\eta(t))dt.\end{split}
\end{equation}
Thus, the optimal feedback strategies do not depend on the concrete choice of the realization of the motion of the representative player.

\section{Optimal control reformulation}\label{sect:optimal_control}
The purpose of the section is to show that the   finite state mean field game interpreted in the sense of Definition \ref{def:control_process} is equivalent to the following control problem:
\begin{equation}\label{criterion:control}
\begin{split}
\text{minimize }J(\phi(\cdot),&m(\cdot),\mu(\cdot),\nu(\cdot))\triangleq \\&\mu_0\phi(t_0)-\mu(T)\sigma(m(T))+\int_{t_0}^T\mu(t)g(t,m(t),\nu(t))dt\end{split}
\end{equation} subject to  $(\phi(\cdot),m(\cdot),\mu(\cdot),\nu(\cdot))$ satisfying 
\begin{equation}\label{eq:dymanics_m_control}
\frac{d}{dt}m(t)=m(t)\mathcal{Q}(t,m(t),\nu(t)), 
\end{equation}
\begin{equation}\label{eq:dynamics_mu_sum}
\frac{d}{dt}\mu(t)=\mu(t)\mathcal{Q}(t,m(t),\nu(t)), 
\end{equation} 
\begin{equation}\label{eq:Bellman_control}
\frac{d}{dt}\phi(t)=-H(t,m(t),\phi(t)),
\end{equation}
\begin{equation}\label{constraints:optimal_control}
m(t_0)=m_0,\, \mu(t_0)=\mu_0,\, \phi(T)=\sigma(m(T)).
\end{equation}

Notice that the state vector for this problem is $(\phi,m,\mu)$, where $\phi\in\rd$, $m,\mu\in \Sigma^d$, when the control parameter is $\nu=(\nu_1,\ldots,\nu_d)\in (\mathcal{P}(U))^d$. One may regard the vector $\phi(t)$ as the upper bound of the rewards; $m(t)$ stands for the distribution of all players;  $\mu(t)$ describes the evolution of fictitious player.

\begin{proposition}\label{prop:J_leq} For every control process $(\phi(\cdot),m(\cdot),\mu(\cdot),\nu(\cdot))$ satisfying~(\ref{eq:dymanics_m_control}),~(\ref{eq:dynamics_mu_sum}),~(\ref{eq:Bellman_control}),~(\ref{constraints:optimal_control})
	$$J(\phi(\cdot),m(\cdot),\mu(\cdot),\nu(\cdot))\geq 0. $$
\end{proposition}
\begin{proof}
	First, let $e^k$ stand for the $k$-th coordinate vector of $\mathbb{R}^{d*}$. Notice that, for any $\mu_0=(\mu_{0,1},\ldots,\mu_{0,d})\in\Sigma^d$, 
	$$\mu_0=\sum_{k=1}^d\mu_{0,k}e^k.$$
	
	Furthermore, let $[t_0,T]\ni t\mapsto \mu^k(t)\in\Sigma^d$ be a solution of the initial value problem
	\begin{equation}\label{eq:dynamics_mu_k}
	\frac{d}{dt}\mu^k(t)=\mu^k(t)\mathcal{Q}(t,m(t),\nu(t)),\ \ \mu^k(t_0)=e^k.
	\end{equation} The trajectory $\mu^k(\cdot)$ describes the probability distribution for the representative player who starts at time $t_0$ from the state $k$. Notice that, if $\mu(\cdot)$ satisfies~(\ref{eq:dynamics_mu_sum}), then
	$$\mu(\cdot)=\sum_{k=1}^d\mu_{0,k}\mu^k(\cdot). $$
	
	Put \begin{equation}\label{intro:J_k}
	\begin{split}
	J_k(\phi(\cdot)&,m(\cdot),\mu^k(\cdot),\nu(\cdot))\\ &\triangleq \phi_k(t_0)-\mu^k(T)\sigma(m(T))-\int_{t_0}^T\mu^k(t)g(t,m(t),\nu(t))dt\\ &= e^k\phi(t_0)-\mu^k(T)\sigma(m(T))-\int_{t_0}^T\mu^k(t)g(t,m(t),\nu(t))dt.
	\end{split}
	\end{equation}
	Notice that the quantity 
	$$\mu^k(T)\sigma(m(T))+\int_{t_0}^T\mu^k(t)g(t,m(t),\nu(t))dt $$ is the reward of the representative player who starts from the state $k$ and uses the strategy $\nu$ in the case when the distribution of all players is given by $m(\cdot)$. Since $\phi_k(\cdot)$ is  equal to the maximal expected reward of the representative player, we deduce that
	\begin{equation}\label{ineq:J_k}
	J_k(\phi(\cdot),m(\cdot),\mu^k(\cdot),\nu(\cdot))\geq 0.
	\end{equation}
	Furthermore, if $\mu(\cdot)$ satisfies~(\ref{eq:dynamics_mu_sum}), then
	\begin{equation}\label{equality:J_J_k}
	J(\phi(\cdot),m(\cdot),\mu(\cdot),\nu(\cdot))= \sum_{k=1}^d\mu_{0,k}J_k(\phi(\cdot),m(\cdot),\mu^k(\cdot),\nu(\cdot)).
	\end{equation} This and~(\ref{ineq:J_k}) give the conclusion of the proposition.
\end{proof}

In the following theorem, we assume that the data of the mean field game $m_0$ and $\sigma$ are fixed.

\begin{theorem}\label{th:equivalence control}
	For the function $t\mapsto (\phi^*(t),m^*(t))\in \rd\times\Sigma^d$ and the control $t\mapsto\nu^*(t)$ the following statements are equivalent.
	\begin{enumerate}[label=(\roman*)]
		\item\label{cond:th_control:mfg} The pair $(\phi^*(\cdot),m^*(\cdot))$  solves the mean field game with initial distribution $m_0$, whereas $\nu^*(\cdot)$ is the corresponding equilibrium feedback strategy.
		\item\label{cond:th_control:control:some} There exists a flow of probabilities $\mu^*(\cdot)$ and an initial distribution $\mu_0^*\in\Sigma^d$ with nonzero coordinates  such that
		$(\phi^*(\cdot),m^*(\cdot),\mu^*(\cdot),\nu^*(\cdot))$ provides the solution of the optimal control problem~(\ref{criterion:control})--(\ref{constraints:optimal_control}) for $\mu_0=\mu_0^*$.
		\item\label{cond:th_control:control:every} For every $\mu^*(\cdot)$ satisfying~(\ref{eq:dynamics_mu_sum}),  the control process
		$(\phi^*(\cdot),m^*(\cdot),\mu^*(\cdot),\nu^*(\cdot))$ is the solution of the optimal control problem~(\ref{criterion:control})--(\ref{constraints:optimal_control}) with $\mu_0=\mu^*(t_0)$.
		\item\label{cond:th_control:zero:some} The triple $(\phi^*(\cdot),m^*(\cdot),\nu^*(\cdot))$  is such that, for some $\mu_0\in \Sigma^d$ with nonzero coordinates and $\mu^*(\cdot)$, the process $(\phi^*(\cdot),m^*(\cdot),\mu^*(\cdot),\nu^*(\cdot))$ satisfies~(\ref{eq:dymanics_m_control})--(\ref{constraints:optimal_control}) and the equality
		$$J(\phi^*(\cdot),m^*(\cdot),\mu^*(\cdot),\nu^*(\cdot))=0.$$
		\item\label{cond:th_control:zero:every} For every $\mu^*(\cdot)$ satisfying~(\ref{eq:dynamics_mu_sum}) and $\mu_0=\mu^*(t_0)$, the control process  $(\phi^*(\cdot),m^*(\cdot),\mu^*(\cdot),\nu(\cdot))$ satisfies~(\ref{eq:dymanics_m_control})--(\ref{constraints:optimal_control}) and the following equality holds:
		$$J(\phi^*(\cdot),m^*(\cdot),\mu^*(\cdot),\nu^*(\cdot))=0.$$
	\end{enumerate}

\end{theorem}
\begin{proof} The scheme of implications proving the desired equivalence can be illustrate by the following commutative  diagram.
	
	\[
	\xymatrix{
		\mathrm{(i)}  \ar[dr] &  \\
		\mathrm{(iv)}  \ar[u] & \ar[l]\ar[d] \mathrm{(v)} \\
		\mathrm{(ii)}  \ar[u] & \ar[l]  \mathrm{(iii)}
	}
	\]
	
	Obviously, we have that~\ref{cond:th_control:control:every} implies~\ref{cond:th_control:control:some}. Analogously,~\ref{cond:th_control:zero:some} follows from~\ref{cond:th_control:zero:every}. 
	
	Now we prove implication~\ref{cond:th_control:zero:every}$\Rightarrow$\ref{cond:th_control:control:every}.
	Let $\mu^{*,k}$, $k=1,\ldots,d$, satisfy~(\ref{eq:dynamics_mu_k}) for $m(\cdot)=m^*(\cdot)$ and $\nu=\nu^*(\cdot)$.
	Condition~\ref{cond:th_control:zero:every} implies that each value $J_k(\phi^*(\cdot),m^*(\cdot),\mu^{*,k}(\cdot),\nu^*(\cdot))=0$. Further, let $\mu_0=(\mu_{0,1},\ldots,\mu_{0,d})$ be an arbitrary element of $\Sigma^d$. The function $\mu^*(\cdot)$ defined by
	\[\mu^*(t)\triangleq \sum_{k=1}^d\mu_{0,k}\mu^{*,k}(t)\] solves~(\ref{eq:dynamics_mu_sum}) with initial condition $\mu^*(t_0)=\mu_0$. Moreover, we have that
	\[J(\phi^*(\cdot),m^*(\cdot),\mu^*(\cdot),\nu^*(\cdot))=\sum_{k=1}^d\mu_{0,k}J_k(\phi^*(\cdot),m^*(\cdot),\mu^{*,k}(\cdot),\nu^*(\cdot))=0.\]
	This and Proposition \ref{prop:J_leq} yield that  $(\phi^*(\cdot),m^*(\cdot),\mu^*(\cdot),\nu^*(\cdot))$ provides the solution of the optimal control problem~(\ref{criterion:control})--(\ref{constraints:optimal_control}) for every $\mu^*(\cdot)$, $\mu_0$ satisfying ODE~(\ref{eq:dynamics_mu_sum}) and initial condition $\mu^*(t_0)=\mu_0$. Thus,~\ref{cond:th_control:zero:every} implies~\ref{cond:th_control:control:every} and, consequently,~\ref{cond:th_control:control:some}. 
	
	By Proposition~\ref{prop:equiv:probability},  statement~\ref{cond:th_control:mfg} is equivalent to the following one:  equations~(\ref{eq:dymanics_m_control}),~(\ref{eq:Bellman_control}) hold true and $J_k(\phi^*(\cdot),m^*(\cdot),\mu^{*,k}(\cdot),\nu^*(\cdot))=0$. Indeed, let $(\Omega,\mathcal{F},\{\mathcal{F}_t\}_{t\in [t_0,T]},P,X^k)$ be the motion of the representative player corresponding to the control $\nu^*(\cdot)$, flow of probabilities $m^*(\cdot)$ and initial distribution $e^k$.  Proposition~\ref{prop:equiv:probability} yields that
	$$\phi_k^*(t_0)=    \mathbb{E}\left[\sigma(X^k(T),m(T))+\int_{t_0}^T g(t,X^k(t),m(t),\nu(t))dt\right], $$ where  $\mu^{*,k}_i=P(X^k(t)=i)$, $\phi^*(t_0)=(\phi^*_k(t_0))_{k=1}^d$. This and equality~(\ref{equality:exp_payoff_mu}) imply that
	$$\phi_k^*(t_0)=\mu^{*,k}(T)\phi^*(T)+\int_{t_0}^T\mu^{*,k}(t)g(t,m^*(t),\nu^*(t))dt. $$ Using the definition of $J_k$, we deduce that optimality of $\nu^*(\cdot)$ is equivalent to the equality $J_k(\phi^*(\cdot),m^*(\cdot),\mu^{*,k}(\cdot),\nu^*(\cdot))=0$. 
	
	Therefore, if~\ref{cond:th_control:mfg} holds, then one can use~(\ref{equality:J_J_k}) and deduce~\ref{cond:th_control:zero:every}. 
	
	Let us prove implication~\ref{cond:th_control:zero:some}$\Rightarrow$\ref{cond:th_control:mfg}. Since \[J(\phi^*(\cdot),m^*(\cdot),\mu^*(\cdot),\nu^*(\cdot))= \sum_{k=1}^d\mu_{0,k}J_k(\phi^*(\cdot),m^*(\cdot),\mu^{*,k}(\cdot),\nu^*(\cdot))\] and \[J_k(\phi^*(\cdot),m^*(\cdot),\mu^{*,k}(\cdot),\nu^*(\cdot))\geq 0,\, k=1,\ldots,d,\] we conclude that  $J_k(\phi^*(\cdot),m^*(\cdot),\mu^{*,k}(\cdot),\nu^*(\cdot))=0$. This together with~(\ref{eq:dymanics_m_control}) and~(\ref{eq:Bellman_control}) is equivalent to~\ref{cond:th_control:mfg}. 
	
	To complete the proof, we prove the implication~\ref{cond:th_control:control:some}$\Rightarrow$\ref{cond:th_control:zero:some}. Assume that $(\phi^*(\cdot),m^*(\cdot),\mu^*(\cdot),\nu^*(\cdot))$ provides the optimal solution of the control problem~(\ref{criterion:control}),~(\ref{eq:dymanics_m_control}),~(\ref{eq:dynamics_mu_sum}),~(\ref{eq:Bellman_control}). From Proposition~\ref{prop:J_leq}, it follows that \[J(\phi^*(\cdot),m^*(\cdot),\mu^*(\cdot),\nu^*(\cdot))\geq 0.\] Furthermore, recall (see Theorem \ref{th:existence}) that,  for the initial distribution $m_0$, there exists at least one solution of the finite state mean field game $(\phi^\natural(\cdot),m^\natural(\cdot))$ with the corresponding equilibrium  strategy $\nu^\natural(\cdot)$. Using the implication~\ref{cond:th_control:mfg}$\Rightarrow$\ref{cond:th_control:zero:every} for the solution of the mean field game $(\phi^\natural(\cdot),m^\natural(\cdot))$ and the equilibrium feedback strategy $\nu^\natural(\cdot)$, we deduce that 
	the infimum of $J(\phi'(\cdot),m'(\cdot),\mu'(\cdot),\nu'(\cdot))$ over the set of 4-tuples obeying~(\ref{eq:dymanics_m_control})--(\ref{eq:Bellman_control}) is not greater than $0$. This leads the equality
	$$J(\phi^*(\cdot),m^*(\cdot),\mu^*(\cdot),\nu^*(\cdot))= 0.$$ This means that~\ref{cond:th_control:control:some} implies~\ref{cond:th_control:zero:some}.
\end{proof}

\begin{corollary} If $m_0$ has nonzero-coordinates, then the following statements are equivalent.
	\begin{itemize}
		\item The pair $(\phi^*(\cdot),m^*(\cdot))$  solves the mean field game with the initial distribution $m_0$, whereas $\nu^*(\cdot)$ is the corresponding equilibrium  feedback strategy.
		\item The process $(\phi^*(\cdot),m^*(\cdot),\nu^*(\cdot))$ is optimal in the problem
		\[\begin{split}\text{minimize }J'(\phi(\cdot),m(\cdot),\nu(\cdot))\triangleq m_0\phi(&t_0)-m(T)\sigma(m(T))\\-&\int_{t_0}^Tm(t)g(t,m(t),\nu(t))dt\end{split}\] subject to ~(\ref{eq:dymanics_m_control}),~(\ref{eq:Bellman_control}) and
		$m(t_0)=m_0$, $\phi(T)=\sigma(m(T))$.
		\item The triple $(\phi^*(\cdot),m^*(\cdot),\nu^*(\cdot))$ satisfies~(\ref{eq:dymanics_m_control}),~(\ref{eq:Bellman_control}), the boundary condition
		$m^*(t_0)=m_0$, $\phi^*(T)=\sigma(m^*(T))$ and the equality
		$J'(\phi^*(\cdot),m^*(\cdot),\nu^*(\cdot))=0$.
	\end{itemize}
\end{corollary}
\begin{proof}
	The desired equivalence follows from Theorem~\ref{th:equivalence control} and the fact that, if $m(\cdot)$ and $\mu(\cdot)$ satisfying~(\ref{eq:dymanics_m_control}) and~(\ref{eq:dynamics_mu_sum}) respectively are such that $m(t_0)=\mu(t_0)$, then $m(\cdot)=\mu(\cdot)$. 
\end{proof}

\section{Dependence of the solution of the mean field game on the initial distribution}\label{sect:dependence} Generally, the solution of the mean field game is nonunique~\cite{Nonuniqueness_bayraktar_zhang}. Thus, it is reasonable to examine the multifunction assigning to an initial time and initial distribution of players a set of expected values of the representative player. This concept generalized the notion of value function briefly discussed in \S \ref{subsect:master_eq}. In this section, we study the value multfunction using the viability and attainability theories.

\begin{definition}\label{def:value_multifunction} We say that $\mathcal{V}:[0,T]\times\Sigma^d\rightrightarrows\rd$ is a value multifunction if, for any $t_0\in [0,T]$, $m_0\in\Sigma^d$, $\phi_0\in \mathcal{V}(t_0,m_0)$, there exists a solution of the mean field game $(\phi(\cdot),m(\cdot))$ such that
	\[\phi(t_0)=\phi_0,\ \ m(t_0)=m_0.\]
\end{definition}

We also consider the maximal value function that is multivalued mapping assigning to each pair $(t_0,m_0)$ the set of all vectors $\phi_0$ such that $\phi_0=\phi(t_0)$, $m_0=m(t_0)$ for some solution of the mean field game $(\phi(\cdot),m(\cdot))$. Denote the maximal value multifunction by~$\mathcal{W}$.

First, let us present the sufficient condition for a given multifunction $\mathcal{V}:[t_0,T]\times\Sigma^d\rightrightarrows\rd$ to be a value multifunction. To this end, we use the viability property for the dynamics in the space $\rd\times\Sigma^d\times\Sigma^d\times\mathbb{R}$ given by equations~(\ref{eq:dymanics_m_control})--(\ref{eq:Bellman_control}) and the following differential equation:
\begin{equation}\label{eq:z_viability}
\begin{split}
\frac{d}{dt}z(t)=-\mu(t)\mathcal{Q}&(t,m(t),\nu(t))\phi(t)\\&-\mu(t)H(t,m(t),\phi(t))-\mu(t)g(t,m(t),\nu(t)).\end{split}
\end{equation}

\begin{definition}\label{def:viability} We say that the set $\mathcal{A}\subset [0,T]\times\rd\times\Sigma^d\times\Sigma^d\times\mathbb{R}$ is viable w.r.t.  the dynamics~(\ref{eq:dymanics_m_control})--(\ref{eq:Bellman_control}),~(\ref{eq:z_viability}) if, for every $s,r\in [0,T]$, $s<r$, $(s,\phi_*,m_*,\mu_*,z_*)\in \mathcal{A}$, there exists a feedback relaxed control $\nu(\cdot)$ and a 4-tuple $(\phi(\cdot),m(\cdot),\mu(\cdot),z(\cdot))$ such that
	\begin{itemize}
		\item $\phi(\cdot)$, $m(\cdot)$, $\mu(\cdot)$, $z(\cdot)$ and $\nu(\cdot)$ satisfy~(\ref{eq:dymanics_m_control})--(\ref{eq:Bellman_control}),~(\ref{eq:z_viability});
		\item $\phi(s)=\phi_*$, $m(s)=m_*$, $\mu(s)=\mu_*$, $z(s)=z_*$;
		\item $(r,\phi(r),m(r),\mu(r),z(r))\in\mathcal{A}$.
	\end{itemize}
\end{definition}

\begin{theorem}\label{th:viability} Assume that, given a multifunction $\mathcal{V}:[0,T]\times\Sigma^d\rightrightarrows\rd$, one can find a closed set $\mathcal{A}\subset [0,T]\times\rd\times\Sigma^d\times\Sigma^d\times\mathbb{R}$ such that 
	\begin{itemize}
		\item $\mathcal{A}$ is viable w.r.t.~(\ref{eq:dymanics_m_control})--(\ref{eq:Bellman_control}),~(\ref{eq:z_viability});
		\item the inclusion $(T,\phi,m,\mu,z)\in\mathcal{A}$ implies that $\phi=\sigma(m)$, $z=0$;
		\item if $\phi_0\in\mathcal{V}(t_0,m_0)$, then $(t_0,\phi_0,m_0,\mu_0,0)\in\mathcal{A}$ for some $\mu_0\in\Sigma^d$ with nonzero coordinates.
	\end{itemize} Then, $\mathcal{V}$ is a value multifunction.
\end{theorem}

The proof of the Theorem relies on the following lemma.
\begin{lemma}\label{lm:z_property}
	If $(\phi(\cdot),m(\cdot),\mu(\cdot),z(\cdot),\nu(\cdot))$ obeys~(\ref{eq:dymanics_m_control})--(\ref{eq:Bellman_control}),~(\ref{eq:z_viability}), then, for every $s,r\in [0,T]$, $s<r$,
	\begin{equation}\label{equality:z_r_s}z(r)-z(s)=\mu(s)\phi(s)-\mu(r)\phi(r)-\int_s^r\mu(t)g(t,m(t),\nu(t))dt. \end{equation} 
\end{lemma}
\begin{proof}
	In fact,
	\begin{equation*}
	\mu(r)\phi(r)-\mu(s)\phi(s)=\int_s^r\frac{d}{dt}\left[\mu(t)\phi(t)\right]dt.
	\end{equation*} Since $\mu(\cdot)$ and $\phi(\cdot)$ satisfy equations~(\ref{eq:dynamics_mu_sum}) and~(\ref{eq:Bellman_control}) respectively, we obtain that 
	$$\frac{d}{dt}[\mu(t)\phi(t)]=\mu(t)\mathcal{Q}(t,m(t),\nu(t))\phi(t)+\mu(t)H(t,m(t),\phi(t)).$$
	This and (\ref{eq:z_viability}) imply  representation~(\ref{equality:z_r_s}). 
	
\end{proof}

\begin{proof}[Proof of Theorem~\ref{th:viability}] Let $(t_0,\phi_0,m_0,\mu_0)$ be such that $$(t_0,\phi_0,m_0,\mu_0,0)\in\mathcal{A}, $$ and let $\mu_0$ have nonzero coordinates. Using the standard viability technique (see \cite[Theorem 3.3.4]{Aubin}), we deduce that there exists $(\phi(\cdot),m(\cdot),\mu(\cdot),z(\cdot),\nu(\cdot))$ satisfying~(\ref{eq:dymanics_m_control})--(\ref{eq:Bellman_control}),~(\ref{eq:z_viability}) and boundary conditions
	$$m(t_0)=m_0,\ \ \phi(t_0)=\phi_0,\ \ \mu(t_0)=\mu_0,\ \ z(t_0)=0, $$
	$$\phi(T)=\sigma(m(T)),\ \ z(T)=0. $$
	Using Lemma~\ref{lm:z_property} and definition of the function $J$ (see~(\ref{criterion:control})), we conclude that
	\begin{equation*}
	\begin{split}
	z(T)&=z(t_0)+\mu(t_0)\phi(t_0)-\mu(T)\phi(T)-\int_{t_0}^T\mu(t)g(t,m(t),\nu(t))dt\\ &=
	z(t_0)+J(m(\cdot),\phi(\cdot),\mu(\cdot),\nu(\cdot)). \end{split}\end{equation*} From the boundary conditions, it follows that 
	$$J(\phi(\cdot),m(\cdot),\mu(\cdot),\nu(\cdot))=0. $$ This means that, for $(\phi(\cdot),m(\cdot),\nu(\cdot))$,  statement~\ref{cond:th_control:zero:some} of Theorem~\ref{th:equivalence control} is in force. By Theorem~\ref{th:equivalence control} $(\phi(\cdot),m(\cdot))$ is a solution of the finite state mean field game. Thus, $\mathcal{V}$ is a value multifunction.
\end{proof}

Now let us characterize the maximal value multifunction $\mathcal{W}$ in the terms of the attainability domain. Denote
$$\hat{\vartheta}\triangleq (1/\sqrt{d},\ldots,1/\sqrt{d})=\frac{1}{\sqrt{d}}(1,\ldots,1)=\frac{1}{\sqrt{d}}\sum_{i=1}^de^k.$$

\begin{theorem}\label{th:set_of_solutions}
	
	Let $\mathcal{A}_*$ be the set of 5-tuples $(s,\phi(s),m(s),\mu(s),z(s))$ such that $s\in [0,T]$, when the 4-tuple $(\phi(\cdot),m(\cdot),\mu(\cdot),z(\cdot))$ satisfies~(\ref{eq:dymanics_m_control})--(\ref{eq:Bellman_control}),~(\ref{eq:z_viability}) for some relaxed strategy $\nu(\cdot)$ and the boundary condition
	$$m(T)\in\Sigma^d,\ \ \phi(T)=\sigma(m(T)),\ \ \mu(T)\in\Sigma^d,\ \ z(T)=0. $$
	Then, given $t_0\in [0,T]$, $m_0\in\Sigma^d$, \[\mathcal{W}(t_0,m_0)=\{\phi_0:(t_0,\phi_0,m_0,\hat{\vartheta},0)\in\mathcal{A}_*\}.\] 
\end{theorem}
\begin{remark}\label{remark:choice}
	In Theorem ~\ref{th:set_of_solutions} one can replace $\hat{\vartheta}$ with an arbitrary row-vector with nonzero coordinates.
\end{remark} The proof of Theorem \ref{th:set_of_solutions} uses the following auxiliary statement.

\begin{lemma}\label{lm:z_nondecrease}
	Assume that $(\phi(\cdot),m(\cdot),\mu(\cdot),z(\cdot),\nu(\cdot))$ satisfies~(\ref{eq:dymanics_m_control})--(\ref{eq:Bellman_control}),~(\ref{eq:z_viability}). Then, the function $t\mapsto z(t)$  is nondecreasing.
\end{lemma}	
\begin{proof}
	Consider the Markov decision problem on $[s,r]$ with the dynamics given by the Markov chain with the Kolmogorov matrix
	$$\mathcal{Q}(t,m(t),\nu(t)) $$ and the reward equal to
	\begin{equation}\label{criterion:s_r}
	\mathbb{E}\left[\phi_{X(r)}(r)+\int_s^rg(t,X(t),m(t),\nu(t))dt\right]. 
	\end{equation}  From dynamic programming principle it follows that that the value function of this problem at $t\in [s,r]$ is equal to $\phi(t)$.
	
	Let  $(\Omega,\mathcal{F},\{\mathcal{F}_t\}_{t\in [t_0,T]},P,X)$ be a motion of the representative player corresponding to the feedback relaxed control $\nu(\cdot)$, the flow of probabilities $m(\cdot)$ and the initial distribution at time $s$ equal to $\mu(s)$. As above, we have that the probability $P(X(t)=i)=\mu_i(t)$ and, thus, obeys~(\ref{eq:dynamics_mu_sum}).  We have that \[
	\begin{split}	\mathbb{E}\Bigl[\phi_{X(r)}(r)+&\int_s^rg(t,X(t),m(t),\nu(t))dt\Bigr]\\&= \mu(r)\phi(r)+\int_{s}^{r}\mu(t)g(t,m(t),\nu(t))dt.\end{split}\] Since $\phi(s)$ is the value at time $s$ for the Markov decision problem with the payoff given by~(\ref{criterion:s_r}) and the dynamics given by the Markov chain $\mathcal{Q}(t,m(t),\nu(t))$ on $[s,r]$, we have that
	$$\mu(s)\phi(s)-\mu(r)\phi(r)-\int_s^r\mu(t)g(t,m(t),\nu(t))dt\geq 0. $$ Combining this with~(\ref{equality:z_r_s}), we obtain that $z(r)\geq z(s) $ when $r>s$.
\end{proof}

\begin{proof}[Proof of Theorem~\ref{th:set_of_solutions}]
	Notice that $\mathcal{A}_*$ is viable with respect to equations~(\ref{eq:dymanics_m_control})--(\ref{eq:Bellman_control}),~(\ref{eq:z_viability}). Hence, by Theorem~\ref{th:viability} the mapping $(t_0,m_0)\rightarrow\{\phi_0:(t_0,\phi_0,m_0,\vartheta,0)\in\mathcal{A}_*\}$ is a value multifunctions. Hence,
	\begin{equation}\label{incl:A_subset_W}
	\{\phi_0:(t_0,\phi_0,m_0,\hat{\vartheta},0)\in\mathcal{A}_*\}\subset\mathcal{W}(t_0,m_0).
	\end{equation} Let us prove the opposite inclusion.
	
	Choose $\phi_0\in\mathcal{W}(t_0,m_0)$. This means that there exists a solution of the mean field game $(\phi(\cdot),m(\cdot))$  such that $\phi(t_0)=\phi_0$, $m(t_0)=m_0$. Let $\nu(\cdot)$ be a equilibrium relaxed feedback strategy corresponding to this solution and let $\mu(\cdot)$ solve
	(\ref{eq:dynamics_mu_sum}) with the initial condition $\mu(t_0)=\hat{\vartheta}$. Further, put
	$$z(s)\triangleq \mu(s)\phi(s)-\mu(T)\sigma(m(T))-\int_{s}^T\mu(t)g(t,m(t),\nu(t))dt. $$ Notice, that
	$z(\cdot)$ satisfies~(\ref{eq:z_viability}) and $z(t_0)=J(m(\cdot),\phi(\cdot),\mu(\cdot),\nu(\cdot))$. By Theorem~\ref{th:equivalence control}, we conclude that $z(t_0)=0$. Additionally, $z(T)=0$. Using Lemma~\ref{lm:z_nondecrease}, we obtain that $z(s)=0$ for every $s\in [t_0,T]$. This implies that $(t_0,m_0,\phi_0,\hat{\vartheta},0)\in\mathcal{A}_*$. Therefore,
	\[\mathcal{W}(t_0,m_0)\subset \{\phi_0:(t_0,\phi_0,m_0,\hat{\vartheta},0)\in\mathcal{A}_*\}.\] This together with~(\ref{incl:A_subset_W}) yields the theorem.
\end{proof}

Let us complete the section by the fact that  each solution of the master equation determines a value multifunction.

Master equation~(\ref{eq:master_usual}) relies on the  measurable feedback strategies.  However, we primary consider relaxed strategies. Thus, we relax the master equation and arrive at the following:   
\begin{equation}\label{eq:master_relaxed}
\frac{\partial}{\partial s}\Phi(s,\mu)+H(t,\mu,\Phi(s,\mu))\in - \operatorname{co}\mathcal{O}(s,\mu,\Phi(s,\mu))\cdot \frac{\partial \Phi}{\partial \mu}(s,\mu).
\end{equation} Here, 
\[\begin{split}\operatorname{co}\mathcal{O}(s,m,\phi)=\Bigl\{m&\mathcal{Q}(t,m,\nu):\\ &\mathcal{Q}(t,m,\nu)\phi+g(t,m(t),\nu)=H(t,m,\phi),\ \ \nu\in(\mathcal{P}(U))^d\Bigr\}.\end{split}\]

\begin{definition}\label{def:solution_master} We say that $\Phi:[0,T]\times\Sigma^d\rightarrow\rd$ is a smooth solution of master equation in the multivalued form (\ref{eq:master_relaxed}) if $\Phi$ is continuously differentiable and satisfies~(\ref{eq:master_relaxed}) for every $s\in [0,T]$ and $\mu\in\Sigma^d$.
\end{definition} Here the derivative w.r.t. measure is understood in the sense of formula (\ref{intro:derivative_Phi_measure}).

\begin{proposition}\label{prop:master} Let $\Phi$ be a smooth solution of  master equation in the multivalued form~(\ref{eq:master_relaxed}). Then,
	the multifunction
	\[\mathcal{V}(t_0,m_0)\triangleq \{\Phi(t_0,m_0)\}\] is a value multifunction.
\end{proposition}
\begin{proof}
	Choose $t_0\in [0,T]$, $m_0\in\Sigma^d$. We shall prove that there exists a solution of the mean field game $(\phi(\cdot),m(\cdot))$ such that $\phi(t_0)=\Phi(t_0,m_0)$, $m(t_0)=m_0$.	
	
	Let $\mathcal{O}^\sharp(s,\mu)$ be the set of row-vectors $\xi\in \operatorname{co}\mathcal{O}(s,\mu,\Phi(s,\mu))$ such that 
	\begin{equation}\label{equality:coice_xi}
	-\xi\frac{\partial\Phi(s,\mu)}{\partial \mu}=\frac{\partial}{\partial s}\Phi(s,\mu)+H(t,\mu,\Phi(s,\mu)).
	\end{equation} Obviously, $\mathcal{O}^\sharp(s,\mu)$ is convex. The continuity of the function $\Phi$ and its derivatives yields that the mapping $[0,T]\times \Sigma^d\mapsto \mapsto\mathcal{O}^\sharp(s,\mu)\subset \mathbb{R}^{d*}$ is upper semicontinuous. In particular, for each $(s,\mu)$, the set $\mathcal{O}^\sharp(s,\mu)$ is compact.
	
	Further, let $m(\cdot)$ solve the differential inclusion
	\begin{equation}\label{incl:Phi}
	\frac{d}{dt}m(t)\in \mathcal{O}^\sharp(t,m(t)),\ \ m(t_0)=m_0.
	\end{equation}
	
	Consider the multivalued mapping \[
	\begin{split}
	t\multimap \Bigl\{\nu\in(\mathcal{P}(U))^d: \frac{d}{dt}&m(t)=m(t)\mathcal{Q}(t,m(t),\nu),\\ \mathcal{Q}(t,m&(t),\nu)\Phi(t,m(t))+g(t,m(t),\nu)=H(t,m(t),\Phi(t,m(t))) \Bigr\}.\end{split}\] By \cite[\S 18.17, Filippov’s Implicit Function Theorem]{Infinite_dimensional_analysis}, this mapping admits a measurable selector. Denote it by $\hat{\nu}(\cdot)$. Thus,
	\begin{equation}\label{eq:master:Kolmogorov}
	\frac{d}{dt}m(t)=m(t)\mathcal{Q}(t,m(t),\hat{\nu}(t)),\ \ m(t_0)=m_0.
	\end{equation}
	\begin{equation}\label{eq:master:optimal}
	\mathcal{Q}(t,m(t),\hat{\nu}(t))\Phi(t,m(t))+g(t,m(t),\hat{\nu}(t))=H(t,m(t),\Phi(t,m(t)))
	\end{equation}
	Furthermore,~(\ref{equality:coice_xi}) and~(\ref{incl:Phi}) yield the equality
	\begin{equation}\label{eq:master:full_derivative}
	-\left(\frac{d}{dt}m(t)\right)\frac{\partial\Phi(s,\mu)}{\partial \mu}\Bigr|_{s=t,\mu=m(t)}=\frac{\partial}{\partial s}\Phi(s,\mu)\Bigr|_{s=t,\mu=m(t)}+H(t,m,\Phi(s,m)).
	\end{equation}
	
	Now, set $\phi(t)\triangleq \Phi(t,m(t))$. Due to~(\ref{eq:master:full_derivative}), we have that
	it satisfies the Bellman equation 
	\begin{equation}\label{eq:master:Bellman}
	\frac{d}{dt}\phi(t)=-H(t,m(t),\phi(t))
	\end{equation} and the boundary conditions
	\begin{equation}\label{eq:cond:master_Bellman}
	\phi(T)=\Phi(T,m(T))=\sigma(m(T)),\ \ \phi(t_0)=\Phi(t_0,m(t_0)). 
	\end{equation}
	The fact that $\hat{\nu}(\cdot)$ is an optimal control for the representative player directly follows from~(\ref{eq:master:optimal}). Combining this with~(\ref{eq:master:Kolmogorov}),~(\ref{eq:master:Bellman}) and~(\ref{eq:cond:master_Bellman}), we conclude that the pair $(\phi(\cdot),m(\cdot))$ is a solution of the mean field game such that $\phi(t_0)=\Phi(t_0,m(t_0))$, $m(t_0)=m_0$. Since we choose  $(t_0,m_0)$  arbitrarily, the multifunction $\mathcal{V}$ is the value mutifunction.
	
\end{proof}

\section*{Appendix. Existence of the solution of the finite state mean field game}\label{sect:appendix}
\setcounter{theorem}{0}
\renewcommand{\thetheorem}{A.\arabic{theorem}}

Below we give the proof of Theorem \ref{th:existence}. It relies on fixed point arguments and the notion of control measures. Notice that the approach based on control measures is equivalent to one involving relaxed feedback strategies. Simultaneously, the set of control measures is compact. This allows to use the fixed point technique. 

Throughout the Appendix we assume that the initial time $t_0\in [0,T]$ and the initial distribution of players $m_0=(m_{0,1},\ldots,m_{0,d})\in \Sigma^d$ are fixed. The notion of the control measures is introduced as follows.

Let $\mathcal{U}$ denote the set of measures $\alpha$ on $[t_0,T]\times U$ compatible with the Lebesgue measure, i.e.,  for any Borel set $\Gamma\subset [t_0,T]$,
\[\alpha(\Gamma\times U)=\lambda(\Gamma), \] where $\lambda$ denotes the Lebesgue measure on $[t_0,T]$.  Within the control theory, elements of $\mathcal{U}$ are often called control measures. The meaning of this term is explained below. We endow the set $\mathcal{U}$ with the topology of narrow convergence, i.e., the sequence $\{\alpha^n\}_{n=1}^\infty\subset \mathcal{U}$ converges to $\alpha\in\mathcal{U}$ iff, for every $f\in C([t_0,T]\times U)$,
\[\int_{[t_0,T]\times U}f(t,u)\alpha^n(d(t,u))\rightarrow \int_{[0,T]\times U}f(t,u)\alpha(d(t,u))\text{ as }n\rightarrow\infty.\] Notice that $\mathcal{U}$ can be regarded as the compact convex subset of the set of all charges on $[t_0,T]\times U$ that is a Banach space.

The link between   elements of $\mathcal{U}$   and weakly measurable functions  is straightforward. If $\nu:[t_0,T]\rightarrow\mathcal{P}(U)$ be a weakly measurable function, then the corresponding measure $\alpha$ is defined by the rule: for $f\in C([t_0,T]\times U)$,
\begin{equation}\label{equality:measure_def}
\int_{[t_0,T]\times U}f(t,u)\alpha(d(t,u))\triangleq\int_{t_0}^T\int_{U}f(t,u)\nu(t,du)dt.
\end{equation} Conversely, given $\alpha\in\mathcal{U}$ by the disintegration theorem (see \cite[78-111]{Dellacherie_Meyer}) there exists a weakly measurable function $\nu$ such that (\ref{equality:measure_def}) holds true. This disintegration is unique almost everywhere, i.e., if $\nu'$ and $\nu'$ are two weakly measurable functions satisfying (\ref{equality:measure_def}), then 
\[\nu'(t,\cdot)=\nu''(t,\cdot)\text{ a.e.}\] Below, we denote the disintegration of the measure $\alpha\in\mathcal{U}$ by $\alpha(\cdot|t)$. 

Using the disintegration, one can give the meaning of the control measures. The feedback formalization implies that, for each state $i$, we choose a control measure $\alpha_i$, while the player occupying the state $i$ at time $s$ shares his/her controls according to $\alpha_i(du|s)$, i.e., the disintegration of the control measure gives the feedback strategy. Thus, the Kolmogorov equation  can be rewritten as follows: \begin{equation}\label{eq:Kolmogrov_measure}
\frac{d}{dt}m_j(t)=m_{0,j}+\int_{[t_0,t]\times U}\sum_{i=1}^dm_i(\tau)Q_{i,j}(\tau,m(\tau),u)\alpha_i(d(\tau,u)).
\end{equation} Here $m_0=(m_{0,1},\ldots,m_{0,d})$ is the initial distribution. If $\alpha=(\alpha_i)_{i=1}^d\in\mathcal{U}^d$ is a sequence of control measures, then we denote the solution of (\ref{eq:Kolmogrov_measure})  by 
$m[\cdot,\alpha]$.

Now, let $\phi[\cdot,\alpha]$ stand for the solution of the Hamiltion-Jacobi equation with the $m(\cdot)=m[\cdot,\alpha]$, i.e.,
\begin{equation}\label{intro:phi_alpha}\frac{d}{dt}\phi[t,\alpha]=-H(t,m[t,\alpha],\phi[t,\alpha]),\ \ \phi[T,\alpha]=\sigma(m[T,\alpha]).\end{equation} Further, if $t\in [t_0,T]$, $m\in \Sigma^d$, $\phi\in\rd$, then denote by $\Xi_i(t,m,\phi)$ the set of elements $u\in U$ maximizing the quantity 
\[\left[\sum_{j=1}^dQ_{i,j}(t,m,u)\phi_j+g(t,i,m,u)\right],\] i.e., from the definition of $H_i$ (see (\ref{intro:H_i})) it follows that, if $u\in\Xi_i(t,m,\phi)$, then
\[\left[\sum_{j=1}^dQ_{i,j}(t,m,u)\phi_j+g(t,i,m,u)\right]=H_i(t,m,\phi).\] Since $U$ is a metric compact, the set $\Xi_i(t,m,\phi)$ is nonempty for each $t$, $m$ and $\phi$. Moreover, if $m(\cdot)$ and $\phi(\cdot)$ are continuous function, then the dependence 
\[t\mapsto\Xi_i(t,m(t),\phi(t))\] is upper semicontinuous. 

When $m(\cdot)=m[\cdot,\alpha]$ and $\phi(\cdot)=\phi[\cdot,\alpha]$, we denote the graph of the mapping
\[t\mapsto \Xi_i(t,m[t,\alpha],\phi[t,\alpha]),\] i.e., we
  set
\begin{equation}\label{intro:graph_K_i}\begin{split}
\mathcal{K}_i[\alpha]\triangleq \Bigl\{(t,u^*)\in &[t_0,T]\times U:\\\sum_{j=1}^d&Q_{i,j}(t,m[t,\alpha],u^*)\phi_j[t,\alpha]+g(t,i,m[t,\alpha],u^*)\\=&\max_{u\in U}\bigl[\sum_{j=1}^dQ_{i,j}(t,m[t,\alpha],u)\phi_j[t,\alpha]+g(t,i,m[t,\alpha],u)\bigr]\Bigr\}.\end{split}\end{equation}

Using this, we can reformulate the definition of the solution of the finite state mean field game in the term of control measures. 
\begin{proposition}\label{prop:equivalence}
	A pair $(\phi(\cdot),m(\cdot))$ is a solution of the mean field game iff there exists a  $\hat{\alpha}=(\hat{\alpha}_1,\ldots,\hat{\alpha}_d)\in\mathcal{U}^d$ such that
	\begin{enumerate}
		\item $m(\cdot)=m[\cdot,\hat{\alpha}]$, $\phi(\cdot)=\phi[\cdot,\hat{\alpha}]$;
		\item for each $i=1,\ldots,d$, $\operatorname{supp}(\hat{\alpha}_i)\subset\mathcal{K}_i[\hat{\alpha}]$.
	\end{enumerate}
\end{proposition}
\begin{proof}
	First, assume that $\phi(\cdot)$, $m(\cdot)$ is the solution of the finite state mean field game in the sense of Definition \ref{def:control_process}. Let $\hat{\nu}=(\hat{\nu}_1,\ldots,\hat{\nu}_d)$ be the corresponding equilibrium strategy. For each $\hat{\nu}_i$, define the control measures $\hat{\alpha}_i$ by (\ref{equality:measure_def}). Thanks to (\ref{eq:Kolmogrov_measure}), we have that
	\[m(\cdot)=m[\cdot,\hat{\alpha}],\ \ \phi(\cdot)=\phi[\cdot,\hat{\alpha}].\] Inclusion (\ref{incl:hat_nu}) and equality (\ref{intro:H_i}) imply that, for a.e. $t\in [t_0,T]$,
	\begin{equation}\label{equality:nu_H_i}\begin{split}
	\int_U\Bigl[\sum_{j=1}^dQ_{i,j}(t,m[t,\hat{\alpha}],u)&\phi_j[t,\hat{\alpha}]+g(t,i,m[t,\hat{\alpha}],u)\\&-H_i(t,m[t,\hat{\alpha}],u)\Bigr]\hat{\nu}_i(t,du)=0.\end{split}\end{equation}
	Integrating this w.r.t. time variable and using (\ref{equality:measure_def}), we obtain that
		\begin{equation}\label{equality:alpha_H_i}\begin{split}
	\int_{[t_0,T]\times U}\Bigl[\sum_{j=1}^dQ_{i,j}(t,m[t,\hat{\alpha}],u)&\phi_j[t,\hat{\alpha}]+g(t,i,m[t,\hat{\alpha}],u)\\&-H_i(t,m[t,\hat{\alpha}],u)\Bigr]\hat{\alpha}(d(t,u))=0.\end{split}\end{equation} This, the definitions of the set $\mathcal{K}_i$ (see (\ref{intro:graph_K_i})) and the function $H_i$ (see (\ref{intro:H_i})) yield that
	\[\operatorname{supp}(\hat{\alpha}_i)\subset\mathcal{K}_i[\hat{\alpha}].\] Therefore, each solution of the finite state mean field type satisfies the properties of this proposition. 
	
	Conversely, assume that $\phi(\cdot)$, $m(\cdot)$ and $\hat{\alpha}=(\hat{\alpha}_1,\ldots,\hat{\alpha}_d)$ satisfy properties 1,2 of the proposition. We define the relaxed controls $\hat{\nu}_i$ to be equal to the disintegration of $\hat{\alpha}_i$. Hence, $\phi(\cdot)$ and $m(\cdot)$ satisfy equations (\ref{eq:motion_relaxed_star}), (\ref{eq:HJ_def}). 
	
	Further, due to (\ref{intro:graph_K_i}), inclusion
	\[\operatorname{supp}(\hat{\alpha}_i)\subset\mathcal{K}_i[\hat{\alpha}]\] implies (\ref{equality:alpha_H_i}). Since $\hat{\nu}_i$ is a disintegration of $\hat{\alpha}_i$, we conclude that (\ref{equality:nu_H_i}) holds true for a.e. $t\in [t_0,T]$. This and definition of $H_i$ (see (\ref{intro:H_i})) give inclusion  (\ref{incl:hat_nu}).
\end{proof}

In the light of Proposition \ref{prop:equivalence}, we reduce the existence theorem for the finite state mean field game to the fixed point problem for the  multivalued mapping $\Phi$ that assigns to $\alpha\in\mathcal{U}^d$ the set of sequence of measures $\beta=(\beta_1,\ldots,\beta_d)\in \mathcal{U}^d$ such that
\[\operatorname{supp}(\beta_i)\subset\mathcal{K}_i[\alpha],\ \ i=1,\ldots,d.\] To apply the fixed point theorem, we prove that $\Phi$ has compact and convex values and is upper semicontinuous.

Since the set $\mathcal{K}_i[\alpha]$ is compact, we have that the set of measures $\beta_i\in\mathcal{U}$ those are supported on $\mathcal{K}_i[\alpha]$ is compact and convex. To prove the upper semicontinuity of the mapping $\Phi$, we, first, prove that $m[\cdot,\alpha]$ and $\phi[\cdot,\alpha]$ depend on $\alpha$ continuously, then apply the fact that the mapping assigning to the  functions $m(\cdot)$, $\phi(\cdot)$ the graph of $\Xi_i(\cdot,m(\cdot),\phi(\cdot))$ is upper semicontinuous.

In the following, we say the sequence $\{\alpha^n\}_{n=1}^\infty\subset \mathcal{U}^d$ converges to $\alpha\in\mathcal{U}^d$ provided that each sequence of measures $\{\alpha_i^n\}_{n=1}^\infty$ converges narrowly to $\alpha_i$, where $\alpha^n=(\alpha_1^n,\ldots,\alpha_d^n)$, $\alpha=(\alpha_1,\ldots,\alpha_d)$.

\begin{lemma}\label{lm:m_continuous} If $\{\alpha^n\}\subset\mathcal{U}^d$ converges to $\alpha\in\mathcal{U}^d$, then  $\{m[\cdot,\alpha^n]\}_{n=1}^\infty$ converges to $m[\cdot,\alpha^*]$ in $C([t_0,T],\Sigma^d)$.
\end{lemma}
\begin{proof} To simplify designation set
	\[m^n(\cdot)\triangleq m[\cdot,\alpha^n],\ \ m^*(\cdot)\triangleq m[\cdot,\alpha^*].\]
	 Recall that $m^n(t)$ is a vector with coordinates $(m_1^n(t),\ldots,m_d^n(t))$, while 
	$m^*(t)=(m^*_1(t),\ldots,m^*_d(t))$. 
	We have that
	\begin{equation}\label{ineq:A_m_n_star}\begin{split}
	 |m_j^n(t)-m^*_j(t)&|\\=\Bigl|\int_{[t_0,t]\times U}&\sum_{i=1}^dm_i^n(\tau)Q_{i,j}(\tau,m^n(\tau),u)\alpha_i^n(d(\tau,u))\\&-\int_{[t_0,t]\times U}\sum_{i=1}^d m_i^*(\tau)Q_{i,j}(\tau,m^*(\tau),u)\alpha_i^*(d(\tau,u))\Bigr|
	 \\\leq \int_{[t_0,t]\times U}&\sum_{i=1}^d|m_i^n(\tau)Q_{i,j}(\tau,m^n(\tau),u)-m_i^*(\tau)Q_{i,j}(\tau,m^*(\tau),u)|\alpha_i^n(d(\tau,u))\\&+\Bigl|
	 \int_{[t_0,t]\times U}\sum_{i=1}^dm_i^*(\tau)Q_{i,j}(\tau,m^*(\tau),u)\alpha_i^n(d(\tau,u))
	 \\ &\hspace{40pt}-
	 \int_{[t_0,t]\times U}\sum_{i=1}^dm_i^*(\tau)Q_{i,j}(\tau,m^*(\tau),u)\alpha_i^*(d(\tau,u))\Bigr|.
	\end{split}
	\end{equation}
	
	Notice that
	\begin{equation}\label{ineq:A:lipschitz}\begin{split}
	\sum_{i=1}^d|&m_i^n(t)Q_{i,j}(t,m^n(t),u)-m_i^*(t)Q_{i,j}(t,m^*(t),u)|\\ &\leq 
	C_0\sum_{i=1}^d|m_i^n(t)-m_i^*(t)|+dC_1\|m^n(t)-m^*(t)\|\leq
	C_2\|m^n(t)-m^*(t)\|.
	\end{split}
	\end{equation} Here $C_0$ is the upper bound of $|Q_{i,j}(t,m,u)|$, $C_1$ is the common Lipschitz contant for the functions $m\mapsto Q_{i,j}(t,m,u)$, while $C_2\triangleq\sqrt{d}C_0+dC_1$.
	
	Further, let $L$ be a positive number. Set
	\[\tau_l^L\triangleq Tl/L.\]
	For each $i,j\in \{1,\ldots,d\}$, $l\in \{0,\ldots,L-1\}$, consider the function $f^L_{i,j,l}:[t_0,T]\times U \mapsto \mathbb{R}$ such that
	\begin{itemize}
		\item $f_{i,j,l}^L(t,u)=m_i^*(t)Q_{i,j}(t,m^*(t),u)$ if $t\in [0,\tau_l^L]$;
		\item $f_{i,j,l}^L(t,u)=0$ if $t\in [\tau^L_{l+1},T]$;
		\item the function $t\mapsto f_{i,j,l}^L(t,u)$ is linear on $[\tau_l^L,\tau_{l+1}^L]$.
	\end{itemize} Notice that, for every $t\in [\tau_l^L,\tau_{l+1}^L]$ and each $\alpha\in\mathcal{U}$, 
	\begin{equation}\label{ineq:A:f_ikl_m}
	\begin{split}
	\Bigl|\int_{[t_0,T]\times U}f_{i,j,l}^L&(t,u)\alpha(d(t,u))\\&-\int_{[t_0,T]\times U} m_i^*(t)Q_{i,j}(t,m^*(t),u)\alpha(d(t,u))\Bigr|\leq C_0T/L.\end{split}
	\end{equation} Since $\{\alpha_i^n\}$ narrowly converge to $\alpha^*_i$, we have that there exists $N_L$ such that, for any $n\geq N_L$, and every $i,j\in \{1,\ldots,d\}$ and $l=\{0,\ldots,L-1\}$,
	\[\Bigl|\int_{[t_0,T]\times U}f_{i,j,l}^L(t,u)\alpha_i^n(d(t,u))-\int_{[t_0,t]\times U}f_{i,j,l}^L(t,u)\alpha_i^n(d(t,u))\Bigr|\leq 1/L.\]
	
	Plugging this, (\ref{ineq:A:lipschitz}) and (\ref{ineq:A:f_ikl_m}) into the right-hand side of (\ref{ineq:A_m_n_star}), we conclude that, for each $k=1,\ldots,d$ and $n\geq N(L)$,
	\[
	|m_j^n(t)-m^*_j(t)|\leq \int_{t_0}^T C_2 \|m^n(\tau)-m^*(\tau)\|d\tau+C_3/L.
	\] Here we denote $C_3\triangleq d(C_0T+1)$.
	Hence, if $n\geq N(L)$,
	\[\|m^n(t)-m^*(t)\|\leq C_2\sqrt{d} \int_{t_0}^T  \|m^n(\tau)-m^*(\tau)\|d\tau+\sqrt{d}C_3/L.\] Applying the Gronwall's inequality we obtain that, for $n\geq N(L)$ and  each $t\in [t_0,T]$,
	\[\|m^n(t)-m^*(t)\|\leq C_4/L,\] where $C_4$ is a constant equal to  $\sqrt{d}C_3\exp(C_2\sqrt{d}T)$. This gives the statement of the lemma.
\end{proof}

Using the continuity of the mapping $\alpha\mapsto m[\cdot,\alpha]$ we are able to derive the following.
\begin{lemma}\label{lm:Phi_upper_semi}
	The mapping $\Phi$ is upper semicontinuous, i.e., if $\{\alpha^n\}_{n=1}^\infty\subset \mathcal{U}^d$ converges to $\alpha^*\in\mathcal{U}^d$, $\{\beta^n\}_{n=1}^\infty\subset{U}^d$ is such that $\beta^n\in\Phi(\alpha^n)$ and $\{\beta^n\}_{n=1}^\infty$ converges to $\beta^*$, then $\beta^*\in\Phi(\alpha^*)$.
\end{lemma}
\begin{proof}
	It suffices to show that each probability $\beta^*_i$ is supported on $\mathcal{K}_[\alpha^*]$, where $\beta^*=(\beta^*_1,\ldots,\beta^*_d)$.
	
	Lemma \ref{lm:m_continuous} states that $m[\cdot,\alpha^n]\rightarrow m[\cdot,\alpha^*]$ in $C([t_0,T];\Sigma^d)$ as $n\rightarrow\infty$. Hence, from the definition of $\phi[\cdot,\alpha]$ (see (\ref{intro:phi_alpha}))  and continuous dependence of the solution of ODE on parameter, we conclude that $\{\phi[\cdot,\alpha^n]\}_{n=1}^\infty$ converges to $\phi[\cdot,\alpha^*]$ when $n\rightarrow\infty$. This and very definition of the set $\mathcal{K}_i[\alpha]$ (\ref{intro:graph_K_i}) yield that, if a sequence $\{(t^n,u_i^n)\}_{n=1}^\infty\subset [0,T]\times U$ converges to some $(t^*,u^*_i)$ and $(t^n,u^n_i)\in\mathcal{K}[\alpha^n]$, then 
	$(t^*,u^*)\in\mathcal{K}[\alpha^*]$. 
	
	Further, recall that $\beta^n=(\beta^n_1,\ldots,\beta^n_d)$, whilst the convergence of $\{\beta^n\}$ to $\beta^*$ means that, for each $i$ $\beta^n_i$ narrowly converges to $\beta^*_i$. Due to \cite[Proposition 5.1.8]{Ambrosio}, we have that, for every $(t^*,u_i^*)\in\operatorname{supp}(\beta^*_i)$, there exists a sequence $\{(t^n,u_i^n)\}$ that converges to $(t^*,u^*_i)$ such that $(t^n,u_i^n)\in\operatorname{supp}(\beta_i^n)$. The construction of $\Phi$ implies that $\operatorname{supp}(\beta^n_i)\subset \mathcal{K}_i[\alpha^n]$. Therefore, using the upper semicontinuity of the  $\mathcal{K}_i$ (see above), we conclude that $(t^*,u^*_i)\in\mathcal{K}_i[\alpha^*]$. This gives the conclusion of the lemma.
\end{proof}

\begin{proof}[Proof of Theorem \ref{th:existence}]
Since the multivalued mapping $\Phi$ has convex and compact values and is upper semicontinuous (see Lemma \ref{lm:Phi_upper_semi}), by the Fan–Glicksberg fixed point theorem (see \cite[Corollary 17.55]{Infinite_dimensional_analysis}) it admits a fixed point. Denote it by $\hat{\alpha}$. Let $\hat{m}(\cdot)\triangleq m[\cdot,\hat{\alpha}]$, $\hat{\phi}(\cdot)\triangleq \phi[\cdot,\hat{\alpha}]$. By Proposition \ref{prop:equivalence}, the pair $(\hat{\phi}(\cdot),\hat{m}(\cdot))$ is the solution of the finite state mean field game in the sense of Definition \ref{def:control_process}.
\end{proof}

\begin{acknowledgement} The article was prepared in the framework of a research grant funded by the Ministry of Science and Higher Education of the Russian Federation (grant ID: 075-15-2020-928).
	\end{acknowledgement}

\bibliography{finite_stat_mfg}

\begin{thebibliography}{10}

\bibitem{Infinite_dimensional_analysis}
C.~D. Aliprantis and K.~C. Border.
\newblock {\em Infinite Dimensional Analysis: A Hitchhiker's Guide}.
\newblock Springer, Berlin, Heidelberg, 2006.

\bibitem{Ambrosio}
L.~Ambrosio, N.~Gigli, and G.~Savar\'{e}.
\newblock {\em Gradient flows: in metric spaces and in the space of probability
  measures}.
\newblock Lectures in Mathematics. ETH Zurich. Birkh{\"a}user, Basel, 2005.

\bibitem{Aubin}
J.-P. Aubin.
\newblock {\em Viability theory}.
\newblock Birkh{\"a}user, Boston, 2009.

\bibitem{Basna_Hilbert_Kolokoltsov}
R.~Basna, A.~Hilbert, and V.~N. Kolokoltsov.
\newblock An approximate {N}ash equilibrium for pure jump {M}arkov games of
  mean-field-type on continuous state space.
\newblock {\em Stochastics}, 89(6-7), 2016.

\bibitem{Bayraktar_Ceccin_et_al_common_noise}
E.~Bayraktar, A.~Cecchin, A.~Cohen, and F.~Delarue.
\newblock Finite state mean field games with {W}right-{F}isher common noise as
  limits of ${N}$-player weighted games.
\newblock Preprint at ArXiv:2012.04845, 2020.

\bibitem{Bayraktar_et_al_Wright_N_player}
E.~Bayraktar, A.~Cecchin, A.~Cohen, and F.~Delarue.
\newblock Finite state mean field games with {W}right-{F}isher common noise.
\newblock {\em J. Math. Pures Appl.}, 147:98--162, 2021.

\bibitem{Bayraktar_Cohen_master}
E.~Bayraktar and A.~Cohen.
\newblock Analysis of a finite state many player game using its master
  equation.
\newblock {\em SIAM J. Control. Optim.}, 56(5):3538–3568, 2018.

\bibitem{Nonuniqueness_bayraktar_zhang}
E.~Bayraktar and X.~Zhang.
\newblock On non-uniqueness in mean field games.
\newblock {\em Proc. Amer. Math. Soc.}, 148:4091--4106, 2020.

\bibitem{finite_state_common_noise}
C.~Belak, D.~Hoffmann, and F.~T. Seifried.
\newblock Continuous-time mean field games with finite state space and common
  noise.
\newblock {\em Appl. Math. Optim.}, 2021.
\newblock accepted.

\bibitem{Master_cdll}
P.~Cardaliaguet, F.~Delarue, J.-M. Lasry, and P.-L. Lions.
\newblock {\em The Master Equation and the Convergence Problem in Mean Field
  Games}.
\newblock Princeton University Press, Princeton, 2019.

\bibitem{Cecchin_Fischer_probabilistic_finite_state}
A.~Cecchin and M.~Fischer.
\newblock Probabilistic approach to finite state mean field games.
\newblock {\em Appl. Math. Opt.}, 81(2):253--300, 2020.

\bibitem{master_Cechin_Pelino}
A.~Cecchin and G.~Pelino.
\newblock Convergence, fluctuations and large deviations for finite state mean
  field games via the master equation.
\newblock {\em Stochastic Process. Appl.}, 129:4510--4555, 2019.

\bibitem{Dellacherie_Meyer}
C.~Dellacherie and P.-A. Meyer.
\newblock {\em Probabilities and Potential}.
\newblock North Holland, Amsterdam, 1979.

\bibitem{Gomes_Mohr_Souza_finite_state}
D.~A. Gomes, J.~Mohr, and R.~R. Souza.
\newblock Continuous time finite state mean field games.
\newblock {\em Appl. Math. Opt.}, 68:99–143, 2013.

\bibitem{Hermandez_Lerma_Guo}
X.~Guo and O.~Hern\'{a}ndez-Lerma.
\newblock {\em Continuous-Time Markov Decision Processes}.
\newblock Springer, New York, 2009.

\bibitem{Huang_Caines_Malhame_2007}
M.~Huang, P.~E. Caines, and R.~P. Malham\'{e}.
\newblock Large-population cost-coupled {L}{Q}{G} problems with nonuniform
  agents: individual-mass behavior and decentralized {N}ash equilibria.
\newblock {\em IEEE Trans. Automat. Control}, 52:1560--1571, 2007.

\bibitem{Huang_Malhame_Caines_2006}
M.~Huang, R.~P. Malham\'{e}, and P.~E. Caines.
\newblock Large population stochastic dynamic games: closed-loop
  {M}c{K}ean-{V}lasov systems and the {N}ash certainty equivalence principle.
\newblock {\em Commun. Inf. Syst.}, 6:221--251, 2006.

\bibitem{Kolokoltsov_security}
S.~Katsikas and V.~N. Kolokoltsov.
\newblock Evolutionary, mean-field and pressure-resistance game modelling of
  networks security.
\newblock {\em J. Dyn. Games}, 6(4):315--335, 2019.

\bibitem{Kolokoltsov_Bensoussan}
V.~Kolokoltsov and A.~Bensoussan.
\newblock Mean-field-game model for botnet defense in cyber-security.
\newblock {\em Appl. Math. Opt.}, 74(3):669–692, 2016.

\bibitem{Kolokoltsov_Li_Yang_2011}
V.~Kolokoltsov, J.~J. Li, and W.~Yang.
\newblock Mean field games and nonlinear {M}arkov processes.
\newblock Preprint at arXiv:1112.3744v2, 2011.

\bibitem{Kolokoltsov_Malofeev_book}
V.~Kolokoltsov and O.~Malafeyev.
\newblock {\em Many agent games in socio-economic systems: corruption,
  inspection, coalition building, network growth, security}.
\newblock Springer Nature, New York, 2019.

\bibitem{Kolokoltsov_Yang}
V.~Kolokoltsov and W.~Yang.
\newblock Inspection games in a mean field setting.
\newblock Preprint at ArXiv:1507.08339, 2015.

\bibitem{Kol_book}
V.~N. Kolokoltsov.
\newblock {\em Nonlinear {M}arkov process and kinetic equations}.
\newblock Cambridge University Press, Cambridge, 2010.

\bibitem{Kolokoltsov_Malofeev_2}
V.~N. Kolokoltsov and O.~A. Malafeyev.
\newblock Corruption and botnet defense: a mean field game approach.
\newblock {\em Int. J. Game Theory}, 47:977--999, 2018.

\bibitem{Lasry_Lions_2006_I}
J.-M. Lasry and P.-L. Lions.
\newblock Jeux \`{a} champ moyen. {I}. {L}e cas stationnaire.
\newblock {\em C. R. Math. Acad. Sci. Paris}, 343:619--625, 2006.

\bibitem{Lasry_Lions_2006_II}
J.-M. Lasry and P.-L. Lions.
\newblock Jeux \`{a} champ moyen. {I}{I}. {H}orizon fini et contr\^{o}le
  optimal.
\newblock {\em C. R. Math. Acad. Sci. Paris}, 343:679--684, 2006.

\bibitem{lions_lecture}
P.-L. Lions.
\newblock College de {F}rance course on mean-field games.
\newblock College de France, 2007-2011.

\end{thebibliography}

\end{document}